\documentclass[12pt]{article}
\usepackage{amsmath}
\usepackage[affil-it]{authblk}
\usepackage{amsfonts}
\usepackage{amsthm}
\usepackage{amssymb}
\usepackage{makecell}
\usepackage{hyperref}
\usepackage{appendix}
\usepackage{algorithmic}
\usepackage{graphicx}
\usepackage[ruled,vlined]{algorithm2e}
\usepackage{tabu}
\usepackage{color,ulem}
\definecolor{darkgreen}{rgb}{0.0, 0.63, 0.0}
\theoremstyle{definition}

\usepackage{mathtools}

\theoremstyle{lemma}
\newtheorem{theorem}{Theorem}

\newtheorem{lemma}[theorem]{Lemma}
\newtheorem{proposition}{Proposition}

\DeclareMathOperator{\Log}{Log}

\newcommand{\vol}{\text{Vol}}
\newcommand{\norm}{\text{Nm}}

\newcommand{\trace}{\text{Tr}}

\usepackage[margin=0.9in]{geometry}

\begin{document}
\title{On Pisot Units and the Fundamental Domain of Galois Extensions of $\mathbb{Q}$}
\author[]{Christian Porter, Alexandre Bali, Alar Leibak}
\maketitle
\begin{abstract}
In this paper, we present two main results. Let $K$ be a number field that is Galois over $\mathbb{Q}$ with degree $r+2s$, where $r$ is the number of real embeddings and $s$ is the number of pairs of complex embeddings. The first result states that the number of facets of the reduction domain (and therefore the fundamental domain) of $K$ is no greater than $O\left(\left(\frac{1}{2}(r+s-1)^\delta(r+s)^{1+\frac{1}{2(r+s-1)}}\right)^{r+s-1}\right) \cdot\left(e^{1+\frac1{2e}}\right)^{r+s}(r+s)!$, where $\delta=1/2$ if $r+s \leq 11$ or $\delta=1$ otherwise. The second result states that there exists a linear time algorithm to reduce a totally positive unary form $axx^*$, such that the new totally positive element $a^\prime$ that is equivalent to $a$ has trace no greater than a constant multiplied by the integer minimum of the trace-form $\trace(axx^*)$, where the constant is determined by the shortest Pisot unit in the number field. This may have applications in ring-based cryptography. Finally, we show that the Weil height of the shortest Pisot unit in the number field can be no greater than $\frac{1}{[K:\mathbb{Q}]}\left(\frac{\gamma}{2}(r+s-1)^{\delta-\frac{1}{2(r+s-1)}}R_K^{\frac{1}{r+s-1}}+(r+s-1)\epsilon\right)$, where $R_K$ denotes the regulator of $K$, $\gamma=1$ if $K$ is totally real or $2$ otherwise, and $\epsilon>0$ is some arbitrarily small constant.
\end{abstract}
\section{Introduction}
Let $K$ be an algebraic field of degree $n=r+2s$ (where $r$ is the number of real embeddings and $s$ is the number of pairs of complex embeddings) over $\mathbb{Q}$ with ring of integers $\mathcal{O}_K$ and unit group $\mathcal{O}_K^*$. We associate to $K$ the canonical embeddings $\sigma_1,\dots,\sigma_{r},\sigma_{r+1},\dots,\sigma_{r+s},\dots,\sigma_{r+2s}$ into $\mathbb{C}$, where $\sigma_{r+s+k}(x)=\sigma_{r+k}(x)^*$ for all $1 \leq k \leq s$ and $*$ denotes the complex conjugate.

Define by $K_{\mathbb{R}}=K \otimes_{\mathbb{Q}} \mathbb{R}$. Note that $K_{\mathbb{R}}=\mathbb{R}^r \times \mathbb{C}^s$. We define the canonical involution $*: K_{\mathbb{R}} \to K_{\mathbb{R}}$ that acts as the identity on $\mathbb{R}^{r}$ and acts as complex conjugation on $\mathbb{C}^{s}$. For any $\alpha=(\alpha_1,\dots,\alpha_{r+s}),\beta=(\beta_1,\dots,\beta_{r+s}) \in K_{\mathbb{R}}$, define $\alpha \beta= (\alpha_1\beta_1,\alpha_2\beta_2,\dots,\alpha_{r+s}\beta_{r+s})$. We say an element $\alpha=(\alpha_1,\dots,\alpha_{r+s}) \in K_{\mathbb{R}}$ is totally positive if every $\alpha_i \in \mathbb{R}^+$. Throughout the paper, we will assume that any number field $K$ that we consider is Galois over $\mathbb{Q}$. In fact, the only lemma that requires this property is Lemma \ref{pisotred}, but unfortunately this lemma is crucial to prove the result of the paper, so we must restrict ourselves to such fields.

Consider
\begin{align*}
\trace(axx^*), \hspace{2mm} x,a \in K_{\mathbb{R}},
\end{align*}
where $a$ is a totally positive element. This generates a real positive-definite quadratic form of dimension $r+s$. We call $axx^*$ a unary form.

We will set $\mathcal{O}_{K_{\mathbb{R}}}$ to be the set of elements $(\sigma_1(x),\dots,\sigma_{r+s}(x))$, where $x \in \mathcal{O}_K$, and $\mathcal{O}_{K_{\mathbb{R}}}^*$ the set of elements $(\sigma_1(u),\dots,\sigma_{r+s}(u))$ such that $u \in \mathcal{O}_K^*$. Then a totally positive element $a$ is said to be reduced if it satisfies
\begin{align}
\trace(a) \leq \trace(avv^*), \label{unitineq}
\end{align}
for all $v \in \mathcal{O}_{K_{\mathbb{R}}}^*$. If $v=(\sigma_1(u),\dots,\sigma_{r+s}(u))$, we use the notation $v^{-1}=(\sigma_1(u^{-1}),\dots,\sigma_{r+s}(u^{-1})) \in \mathcal{O}_{K_\mathbb{R}}^*$. We say that two totally positive elements $a,a^\prime$ are equivalent if $a^\prime = avv^*$ for some $v \in \mathcal{O}_{K_{\mathbb{R}}}^*$. Note then that since $a=a^\prime v^{-1}{v^{-1}}^*$, $a$ can be considered by its equivalence class, where equivalence is determined by multiplying $a$ by $vv^*$ where $v \in \mathcal{O}_{K_\mathbb{R}}^*$, and so the real quadratic forms $\trace(axx^*),\trace(a^\prime xx^*)$ are equivalent. The reduction domain of $K_{\mathbb{R}}$, denoted $\mathcal{F}_{K_{\mathbb{R}}}$, is the set of all reduced totally positive elements of $K_{\mathbb{R}}$, and so clearly every positive element is equivalent to an element in $\mathcal{F}_{K_\mathbb{R}}$. Note that the reduction domain is a fundamental domain for the set of totally positive elements of $K$. The reduction domain is known to be the union of finitely many perfect cones (\cite[Satz 4]{koe60}). In \cite{perfectformsupperbound}, an upper bound on the number of perfect unary forms in any given totally real number field was determined. The facets of the cone $\mathcal{F}_K$ are defined by the inequalities \ref{unitineq}, and it is known that the number of facets of the reduction domain are finite in number, meaning that only finitely many inequalities need to be satisfied in order to determine whether or not a unary form is reduced.

Let $x \in \mathcal{O}_K$ be an algebraic integer of $K$. We say that $x$ is a Pisot-Vijayaraghavan number (shortened to a Pisot number) if the absolute value of $x$ is greater than $1$, but the absolute value of all its Galois conjugates, except for the conjugates that correspond to complex conjugation, have absolute value less than $1$. We say that $x$ is a Pisot unit if $x \in \mathcal{O}_K^*$.

By Dirichlet's unit theorem, we know that the rank of $\mathcal{O}_K^*$ is $r+s-1$. Suppose then that $\mathcal{O}_K^*$ is multiplicatively generated by the elements $u_1,u_2,\dots,u_{r+s-1}$ and $\zeta$ where $\zeta$ is some root of unity in $\mathcal{O}_K$. Consider the logarithmic embedding:

\begin{align*}
\Log: K \to \mathbb{R}^{r+s}: \Log(x)=&(\log(|\sigma_1(x)|),\log(|\sigma_2(x)|),\dots, \log(|\sigma_{r}(x)|),\\&2\log(|\sigma_{r+1}(x)|),2\log(|\sigma_{r+2}(x)|),\dots,2\log(|\sigma_{r+s}(x)|)).
\end{align*}
Then note that under the logarithmic embedding, $\Lambda_K :=\Log(\mathcal{O}_K^*)$ generates a lattice in the space
\begin{align}\label{halfplane}
V=\left\{(x_1,\dots,x_{r+s}) \in \mathbb{R}^{r+s}: \sum_{i=1}^{r+s}x_i=0\right\}.
\end{align}
Let $\| \mathbf{v}\|$ be the $l_p$ norm of an element $\mathbf{v} \in \mathbb{R}^{r+s}$. Throughout the paper, we will use the notation
\begin{align*}
\rho^p (\Lambda)= \max_{x \in W} \min_{v \in \Lambda_K} \|x-v\|_p
\end{align*}
for any lattice $\Lambda$, where $W$ is the space in which $\Lambda$ is full-rank. We will also make use of the notation
\begin{align*}
\lambda_{i,K}(\Lambda)=\min\{c \in \mathbb{R}_{\geq 0}: \dim(\Lambda \cap cK)=i\},
\end{align*}
which is called the $i$th successive minima of $\Lambda$, and $K$ is some $0$-symmetric convex body. We will use the notation $\lambda_{i,p}(\Lambda)$ to denote the $i$th successive minima of $\Lambda$ with respect to the convex body drawn out by the $l_p$-norm.

The aim of this paper is to prove the following results.
\begin{theorem}\label{maintheorem}
Let $N_K$ denotes the number of facets of the reduction domain of $K_{\mathbb{R}}$. Then
\begin{align*}
N_K < &2(r+s-1)\\&+2\left(\frac{1}{2}(r+s-1)^\delta(r+s)^{1-\frac{1}{2(r+s-1)}}+\frac{r+s-1}{2}\frac{\log\left(\frac{r+s+1}{r+s-1}\right)}{R_K^{\frac{1}{r+s-1}}}+\frac{\log\left(\frac{r+s-1}{2}\right)}{R_K^{\frac{1}{r+s-1}}}\right)^{r+s-1}\\& \cdot (r+s)\sum_{k=0}^{\lfloor \frac{r+s}{2} \rfloor}(-1)^k {r+s \choose k}\left(\frac{r+s}{2}-k\right)^{r+s-1},
\end{align*}
where
\begin{align*}
\delta= \begin{cases} 1/2, \hspace{2mm}  &1 \leq r+s-1 \leq 10,
\\ 1, \hspace{2mm} &10<r+s-1.
\end{cases}
\end{align*}
and $R_K$ is the regulator of $K$.
\end{theorem}
Using Lemma \ref{sumbound}, and the fact that $R_K>0.2052\dots$ for all number fields $K$ \cite{friedman}, this gives us
\begin{align*}
N_K < O\left(\left(\frac{1}{2}(r+s-1)^\delta(r+s)^{1-\frac{1}{2(r+s-1)}}\right)^{r+s-1}\right) \cdot\left(e^{1+\frac1{2e}}\right)^{r+s}(r+s)!.
\end{align*}
Numerical evidence seems to suggest that the term $\left(e^{1+\frac1{2e}}\right)^{r+s}$ can also be dropped in the expression above.

The logarithmic Weil height of an algebraic number $x$ in $K$ is defined by
\begin{align*}
h(x)=\frac{1}{[K:\mathbb{Q}]}\left(\sum_{i=1}^r\log^+|\sigma_i(x)|+2\sum_{i=1}^s\log^+|\sigma_i(x)|\right),
\end{align*}
where $\log^+(\alpha)=\max\{\log(\alpha),0\}$. We also prove the following interesting proposition, bounding the height of the Pisot unit with the smallest Weil height.
\begin{proposition}
For all $\epsilon>0$, there exists a Pisot unit with Weil height $h(u)$ satisfying
\begin{align*}
h(u) \leq \frac{1}{[K:\mathbb{Q}]}\left(\frac{\gamma}{2}(r+s-1)^{\delta-\frac{1}{2(r+s-1)}}R_K^{\frac{1}{r+s-1}}+(r+s-1)\epsilon\right),
\end{align*}
where $\delta$ is defined as before, and $\gamma=1$ if $K$ is totally real, or $\gamma=2$ if $K$ is totally complex.
\end{proposition}
Using ideas outlined in this work, we also provide an algorithm to reduce unary forms, which could have applications in ring-based cryptography (see e.g. \cite{cramer}, \cite{euclidean}, \cite{por21}).
\begin{theorem}
For any totally positive element $a \in K_\mathbb{R}$, define by
\begin{align*}
\mu(a) \triangleq \min_{x \in \mathcal{O}_{K_\mathbb{R}} \setminus \{0\}} \trace(axx^*).
\end{align*}
Then given a totally positive element $a$, a Pisot unit $u$ and some parameter $\min_{j \neq 1} |\sigma_j(u)|^2< \delta <1$, there exists an algorithm that computes an equivalent element $a^\prime$ such that
\begin{align}
\trace(a^\prime) \leq \max\left\{\frac{t_K(u,\delta)^2}{\min_{x \in \mathcal{S}} \trace(xx^*)},1\right\}\mu(a), \label{ineq13}
\end{align}
where 
\begin{align*}
t_K(u,\delta)=\sqrt{1+\frac{|u|^2-\delta}{\delta-\max_{j \neq 1}|\sigma_j(u)|^2}},
\end{align*}
and $\mathcal{S}$ denotes the elements of $\mathcal{O}_{K_\mathbb{R}}$ that do not correspond to roots of unity or zero in $\mathcal{O}_K$. Also, if
\begin{align*}
\mu(a)=\trace(axx^*)
\end{align*}
for some $x \in \mathcal{O}_{K_\mathbb{R}}$, then
\begin{align*}
\trace(xx^*) \leq t_K(u,\delta)^2. 
\end{align*}
Moreover, the algorithm takes at most $\mathcal{O}(\log(X)((r+s+1)\log(X)+(r+s)\log(\max_i |\sigma_i(u)|)))$ bit operations, where $X=\max_i a_i$, where $a=(a_1,\dots,a_{r+s})$.
\end{theorem}
\section{Reduction of Unary Forms Via Pisot Units, and Some Useful Lemmas}
\begin{proposition}\label{pisotbound}
If $K$ is not either $\mathbb{Q}$ or imaginary quadratic (i.e. has a nontrivial unit group), then for any $\varepsilon>0$, there exists a Pisot unit $u$ such that
\begin{align*}
 e^{(r+s-2)\rho^{\infty}(\Lambda_K)+(r+s-1)\varepsilon} \leq &|u| \leq e^{(r+s)\rho^{\infty}(\Lambda_K)+(r+s-1)\varepsilon},
\\e^{-(2\rho^{\infty}+\varepsilon)} \leq &|\sigma_i(u)| \leq e^{-\varepsilon},
\end{align*}
for all $i \neq 1$.
\end{proposition}
\begin{proof}
If $u$ is a Pisot unit, then clearly $\log(|u|)>0$ and $\log(|\sigma(u)|)<0$ for any Galois conjugate $\sigma$ that does not correspond to complex conjugation. Let $x=(x_1,\dots,x_{r+s})$ be an element of $V$. By definition, $\sum_{i=1}^{r+s}x_i=0$, and
\begin{align*}
\min_{w \in \Lambda_K}\|x-w\|_{\infty}=\|x-v\|_{\infty}=\max_{1 \leq i \leq r+s}|x_i-v_i| \leq \rho^{\infty}(\Lambda_K),
\end{align*}
for some appropriate $v=(v_1,\dots,v_{r+s}) \in \Lambda_K$. Set $x_1=(r+s-1)\rho^{\infty}(\Lambda_K)+(r+s-1)\varepsilon,x_2=x_3=\dots=x_{r+s}=-\rho^{\infty}(\Lambda_K)-\varepsilon$ for some $\varepsilon>0$. Then we must have
\begin{align*}
&|(r+s-1)\rho^{\infty}(\Lambda_K)-v_1+(r+s-1)\varepsilon| \leq \rho^{\infty}(\Lambda_K),
\\&|v_2+\rho^{\infty}(\Lambda_K)+\varepsilon| \leq \rho^{\infty}(\Lambda_K),
\\&|v_3+\rho^{\infty}(\Lambda_K)+\varepsilon| \leq \rho^{\infty}(\Lambda_K),
\\& \vdots
\\& |v_{r+s}+\rho^{\infty}(\Lambda_K)+\varepsilon| \leq \rho^{\infty}(\Lambda_K).
\end{align*}
Clearly $\rho^{\infty}(\Lambda_K)$ is nonzero if $K$ is not either $\mathbb{Q}$ or imaginary quadratic, so these inequalities yield the following inequalities:
\begin{align*}
0<(r+s-2)\rho^{\infty}(\Lambda_K)+(r+s-1)\varepsilon \leq &v_1 \leq (r+s)\rho^{\infty}(\Lambda_K)+(r+s-1)\varepsilon,
\\ -2\rho^{\infty}(\Lambda_K)-\varepsilon \leq &v_2 \leq -\varepsilon<0,
\\ -2\rho^{\infty}(\Lambda_K)-\varepsilon \leq &v_3 \leq -\varepsilon<0,
\\& \vdots
\\  -2\rho^{\infty}(\Lambda_K)-\varepsilon \leq &v_{r+s} \leq -\varepsilon<0,
\end{align*}
which proves the lemma.
\end{proof}
\begin{lemma}\label{pisotred}
Suppose that $a=(a_1,a_2,\dots,a_{r+s}) \in K_{\mathbb{R}}$ is a totally positive element. Let $u \in \mathcal{O}_K^*$ be a Pisot unit and suppose that
\begin{align*}
\trace(av_iv_i^*) \geq \trace(a),
\end{align*}
for all $1 \leq i \leq r+2s$ where $v_i$ is the element of $K_{\mathbb{R}}$ obtained by embedding $\sigma_i(u)$ into $K_{\mathbb{R}}$. Let
\begin{align*}
t_K(u)=\sqrt{1+\frac{|u|^2-1}{1-\max_{2 \leq j \leq r+s}|\sigma_j(u)|^2}}.
\end{align*}
Then for any $x \in K_{\mathbb{R}}$ satisfying $\trace(xx^*) \geq t_K(u)^2$, $\trace(axx^*) \geq \trace(a)$.
\end{lemma}
\begin{proof}
By assumption, for each $1 \leq i \leq r+2s$,
\begin{align*}
&\trace(av_iv_i^*)=\left(\sum_{j=1}^{r}+ 2\sum_{j=r+1}^{r+s}\right) a_j|\sigma_j(\sigma_i(u))|^2 \geq \trace(a) = \left(\sum_{j=1}^{r}+ 2\sum_{j=r+1}^{r+s}\right) a_j
\\& \iff \left(\sum_{j=1}^{r}+ 2\sum_{j=r+1}^{r+s}\right)a_j(|\sigma_j(\sigma_i(u))|^2-1) \geq 0.
\end{align*}
Suppose that $|\sigma_k(\sigma_i(u))|=|u|$ for some value of $k$. Then
\begin{align*}
&(|u|^2-1)a_k \geq \left(\sum_{j \neq k=1}^{r}+ 2\sum_{j \neq k=r+1}^{r+s}\right)a_j \geq (1-\max_{j \neq 1}|\sigma_j(u)|^2)\left(\sum_{j \neq k=1}^{r}+ 2\sum_{j \neq k=r+1}^{r+s}\right)a_j
\\&\iff \trace(a) \leq t_K(u)^2 a_k.
\end{align*}
By cycling through all values $1 \leq i \leq r+s$, we attain the above inequality for all $1 \leq k \leq r+s$, and so if $x=(x_1,\dots,x_{r+s}) \in K_{\mathbb{R}}$ and $\trace(xx^*) \geq t_K(u)^2$,
\begin{align*}
\trace(axx^*)=\sum_{j=1}^{r+s} a_j |x_j|^2 \geq t_K(u)^{-2}\trace(a)\sum_{j=1}^{r+s} |x_j|^2=t_K(u)^{-2}\trace(a)\trace(xx^*) \geq \trace(a),
\end{align*}
as required.
\end{proof}
\begin{lemma}\label{coveringradiusbound}
Let $\rho^{\infty}(\Lambda)$ denote the covering radius in the $l_{\infty}$ norm of a lattice $\Lambda$ of rank $n$ with volume $\vol(\Lambda)$, and also assume that $\Lambda$ is well-rounded (that is, the successive minima are all equal in value). Then 

\begin{align}
\rho^{\infty}(\Lambda) \leq \begin{cases} \frac{\sqrt{n}}{2}\vol(\Lambda)^{\frac{1}{n}}, \hspace{2mm}  &1 \leq n \leq 10,
\\ \frac{n}{2} \vol(\Lambda)^{\frac{1}{n}}, \hspace{2mm} &10<n. 
\end{cases}
\end{align}
\end{lemma}
\begin{proof}
To prove the first inequality, note first that $\rho^{\infty}(\Lambda) \leq \rho^2(\Lambda)$. It was shown in \cite{K2} that $\rho^2(\Lambda) \leq \frac{\sqrt{n}}{2}\vol(\Lambda)$ for all rank $n$ lattices, where $n \leq 10$.

For the second case, note first that $\lambda_{i,\infty}(\Lambda) \leq \lambda_{i,2}(\Lambda) \leq \sqrt{n}\lambda_{i,\infty}(\Lambda)$. It is well-known that
\begin{align*}
\rho^2(\Lambda) \leq \frac{\sqrt{n}}{2}\lambda_{n,2}(\Lambda),
\end{align*}
and since $\Lambda$ is assumed to be well-rounded, we get
\begin{align*}
\rho^{\infty}(\Lambda) \leq \rho^2(\Lambda) \leq \frac{\sqrt{n}}{2}\lambda_{n,2}(\Lambda) \leq \frac{n}{2}\lambda_{n,\infty}(\Lambda)=\frac{n}{2}\lambda_{1,\infty}.
\end{align*}
It is also well-known that $\lambda_{1,\infty} \leq \vol(\Lambda)^{\frac{1}{n}}$ for any lattice $\Lambda$, and so the second inequality holds.
\end{proof}
\begin{lemma}[\cite{mathstackexchange}] \label{cubevolume}
Let $\mathcal{C}(R)$ denote the $n$-dimensional hypercube of side-length $R$, and let $V$ denote the $n-1$-dimensional half-plane as in \ref{halfplane}. Then
\begin{align*}
\vol(\mathcal{C}(R) \cap V)=\frac{R^{n-1}\sqrt{n}}{(n-1)!}\sum_{k=0}^{\lfloor \frac{n}{2}\rfloor}(-1)^k {n \choose k} \left(\frac{n}{2}-k\right)^{n-1}.
\end{align*}
\end{lemma}
\begin{lemma}[\cite{blichfeldt}] \label{blichfeldt}
Let $K$ be a convex $0$-symmetric body of rank $n$. Then
\begin{align*}
|K \cap \mathbb{Z}^n| \leq n!\vol(K)+n.
\end{align*}
\end{lemma}
\section{Determining an Upper Bound on the Number of Facets of $\mathcal{F}_K$}
\begin{proof}[Proof of Theorem \ref{maintheorem}]
    By Lemma \ref{pisotred}, any totally positive element $a=(a_1,\dots,a_{r+s}) \in \mathcal{F}_K$ has
    \begin{align*}
\trace(axx^*) \geq \trace(a),
    \end{align*}
for all $x \in K_{\mathbb{R}}, \trace(xx^*) \geq t_K^{2}$, where $t_K=\min_{u \in \mathcal{P}}t_K(u)$ and $\mathcal{P}$ is the set of all Pisot units of $K$. Therefore, if $u \in \mathcal{O}_{K_{\mathbb{R}}}^*$ determines a facet of the reduction domain it must satisfy
\begin{align*}
\trace(uu^*) \leq t_K^2 \iff \log(\trace(uu^*)) \leq 2\log(t_K).
\end{align*}
Let $\epsilon(u)=\max_i\{|\sigma_i(u)|,1/|\sigma_i(u)|: 1 \leq i \leq r+s\}$. Then exactly half of the (nontrivial) unit group satisfy $\epsilon(u)=\max_i |\sigma_i(u)|$ and exactly half of them satisfy $\epsilon(u)=\max_i |\sigma_i(u)|^{-1}$, since if $\epsilon(u)=\max_i |\sigma_i(u)|$ then $\epsilon(1/u)=\max_i|\sigma_i(u)|^{-1}$. We consider only units that satisfy $\epsilon(u)=|\sigma_i(u)|$, so
\begin{align*}
2\log(t_K) \geq \log(\trace(uu^*)) \geq 2\|\Log(u)\|_{\infty},
\end{align*}
so the only possible units satisfying $\epsilon(u)=\max_i|\sigma_i(u)|$ that can possibly constitute facets of the reduction domain must satisfy $\|\Log(u)\|_{\infty} \leq \log(t_K)$. Hence, we want to determine the number of lattice points of $\Lambda_K$ are contained within the convex body
\begin{align}
C=\left\{x \in \mathbb{R}^{r+s}: \|x\|_{\infty} \leq \log(t_K)\right\}. \label{cube}
\end{align}
Now, since $\Lambda_K$ is of full-rank in the half plane $V$ as described in \ref{halfplane}, we need to consider the $0$-symmetric convex body $C \cap V$. By Lemma \ref{cubevolume}, this shape has volume equal to
\begin{align*}
\vol(C \cap V) = \frac{\log(t_K)^{r+s-1}\sqrt{r+s}}{(r+s-1)!}\sum_{k=0}^{\lfloor \frac{r+s}{2} \rfloor}(-1)^k {r+s \choose k}\left(\frac{r+s}{2}-k\right)^{r+s-1}.
\end{align*}
We apply the transform that takes $\Lambda_K$ to the set $\mathbb{Z}^{r+s-1}$, rotated in $\mathbb{R}^{r+s}$ so that it sits within the half-plane $V$. Then applying a similar transform to the convex body $C \cap V$ gives us the new convex body $K$ which has volume
\begin{align*}
\vol(K)=\frac{\log(t_K)^{r+s-1}(r+s)}{R_K(r+s-1)!}\sum_{k=0}^{\lfloor \frac{r+s}{2} \rfloor}(-1)^k {r+s \choose k}\left(\frac{r+s}{2}-k\right)^{r+s-1},
\end{align*}
where $R_K$ is the regulator of the field $K$, using the fact that $\vol(\Lambda_K)=R_K/\sqrt{r+s}$ (see \cite{neukirch}). Then by Lemma \ref{blichfeldt}, the number of integer lattice points inside $K$ is upper bounded by
\begin{align*}
r+s-1+\frac{\log(t_K)^{r+s-1}(r+s)}{R_K}\sum_{k=0}^{\lfloor \frac{r+s}{2} \rfloor}(-1)^k {r+s \choose k}\left(\frac{r+s}{2}-k\right)^{r+s-1}.
\end{align*}
It remains to prove a bound on $\log(t_K)$. Clearly, since $\max_{j \neq 1}|\sigma_j(u)|^2<1$ for any Pisot unit $u$, 
\begin{align*}
1<\frac{1}{1-\max_{j \neq 1}|\sigma_j(u)|^2},
\end{align*}
so
\begin{align*}
t_K=\min_{u \in \mathcal{P}}\sqrt{1+\frac{|u|^2-1}{1-\max_{j \neq i}|\sigma_j(u)|^2}}<\min_{u \in \mathcal{P}}\sqrt{\frac{|u|^2}{1-\max_{j\neq 1}|\sigma_j(u)|^2}},
\end{align*}
where $\mathcal{P}$ denotes the set of Pisot units of $K$. Hence by Lemma \ref{pisotbound}, we must have
\begin{align*}
\log(t_K) < (r+s)\rho^{\infty}(\Lambda)+(r+s-1)\epsilon-\log(1-e^{-2\epsilon}),
\end{align*}
for any $\epsilon>0$. The right-hand side of the above inequality attains its minimum at $\epsilon=\frac{1}{2}(\log(r+s+1)-\log(r+s-1))$, for which we get
\begin{align*}
\log(t_K)<(r+s)\rho^{\infty}(\Lambda)+\frac{r+s-1}{2}\log\left(\frac{r+s+1}{r+s-1}\right)+\log\left(\frac{r+s-1}{2}\right).
\end{align*}
Finally, by Lemma \ref{coveringradiusbound} and the fact that the log-unit lattice is well-rounded with respect to the infinity norm (since $K$ is Galois, if say $\Log(u)$ is the minimum vector, then $\Log(\sigma_i(u))$ has identical length with respect to the infinity norm, and there are $r+s-1$ linearly independent vectors of this form), 
\begin{align*}
\rho^{\infty}(\Lambda_K) \leq \begin{cases} \frac{\sqrt{r+s-1}}{2\sqrt{r+s}^{\frac{1}{r+s-1}}}R_K^{\frac{1}{r+s-1}}, \hspace{2mm}  &1 \leq r+s-1 \leq 10,
\\ \frac{(r+s-1)}{2\sqrt{r+s}^\frac{1}{r+s-1}} R_K^{\frac{1}{r+s-1}}, \hspace{2mm} &10<r+s-1. 
\end{cases}
\end{align*}
Then since we have counted exactly half of the required integer lattice points, we get
\begin{align*}
N_K < &2(r+s-1)\\&+2\left(\frac{1}{2}(r+s-1)^\delta(r+s)^{1-\frac{1}{2(r+s-1)}}+\frac{r+s-1}{2}\frac{\log\left(\frac{r+s+1}{r+s-1}\right)}{R_K^{\frac{1}{r+s-1}}}+\frac{\log\left(\frac{r+s-1}{2}\right)}{R_K^{\frac{1}{r+s-1}}}\right)^{r+s-1}\\& \cdot (r+s)\sum_{k=0}^{\lfloor \frac{r+s}{2} \rfloor}(-1)^k {r+s \choose k}\left(\frac{r+s}{2}-k\right)^{r+s-1},
\end{align*}
where
\begin{align*}
\delta= \begin{cases} 1/2, \hspace{2mm}  &1 \leq r+s-1 \leq 10,
\\ 1, \hspace{2mm} &10<r+s-1.
\end{cases}
\end{align*}
\end{proof}
\subsection{The Special Case of $[K:\mathbb{Q}]=3$}
The case where $[K:\mathbb{Q}]=3$ can be treated separately, as in this case, every element of the unit group is either $\pm 1$, a Pisot unit, the inverse of a Pisot unit, the conjugate of a Pisot unit or the inverse conjugate of a Pisot unit. We begin by proving the following useful lemma.
\begin{lemma}\label{n=3lemma}
Suppose that $\mathbf{b}_1=(\alpha,\beta,-\alpha-\beta)$ and $\mathbf{b}_2=(\gamma,\delta,-\gamma-\delta)$ for some $\alpha,\beta,\gamma,\delta \in \mathbb{R}$ satisfying $|\alpha|,|\beta|,|\gamma|,|\delta|, |\alpha+\beta|, |\gamma+\delta| >0$, and assume that $\mathbf{b}_1,\mathbf{b}_2$ are linearly independent over $\mathbb{R}$. Suppose that $\mathbf{b}_1,\mathbf{b}_2$ satisfy
\begin{align}
\|\mathbf{b}_1\|_\infty \leq \|\mathbf{b}_2\|_\infty \leq \|\mathbf{b}_1 \pm \mathbf{b}_2\|_\infty. \label{ineq1}
\end{align}
Let $\mathbf{v}=x\mathbf{b}_1+y\mathbf{b}_2$, for $x,y \in \mathbb{Z}$. Then unless $(|x|,|y|)$ are in the following set:
\begin{align}
S=\{(0,0),(1,0),(0,1),(1,1),(2,1),(1,2)\}, \label{set1}
\end{align}
$\|\mathbf{v}\|_{\infty} \geq 2\lambda^{\infty}$.
\end{lemma}
\begin{proof}
See appendix.
\end{proof}
Now, let $K$ be a number field of degree $3$ over $\mathbb{Q}$. Suppose that the vectors $\mathbf{b}_1,\mathbf{b}_2$ generate $\Log(\mathcal{O}_K^*)$, and suppose that $\mathbf{b}_1,\mathbf{b}_2$ satisfy \ref{ineq1} without loss of generality. Let $u$ denote the unit that corresponds to the shortest non-zero element of $\Log(\mathcal{O}_K^*)$, under the logarithmic embedding. We may assume that $u$ is a Pisot unit, as otherwise $u$ is either the inverse or a conjugate (or both) of a Pisot unit, which does not affect the length of the element under the logarithmic embedding. Note then that
\begin{align*}
t_K(u)=\sqrt{1+\frac{|u|^2-1}{1-\max_{j \neq 1}|\sigma_j(u)|^2}} <\sqrt{1+\frac{|u|^2-1}{\prod_{j \neq 1}(1-|\sigma_j(u)|^2)}}.
\end{align*}
Since $K$ is Galois, $K$ must be totally real. Clearly $1-u^2$ is an algebraic integer, so $|\norm_{K/\mathbb{Q}}(1-u^2)| \geq 1$, which gives
\begin{align*}
|\norm_{K/\mathbb{Q}}(1-u^2)|=(u^2-1)\prod_{j \neq 1}(1-\sigma_j(u)^2) \geq 1 \iff \prod_{j \neq 1}(1-\sigma_j(u)^2) \geq (u^2-1)^{-1},
\end{align*}
and so
\begin{align*}
t_K(u) < \sqrt{1+(u^2-1)^2}.
\end{align*}
Assume without loss of generality that
\begin{align*}
\trace(au^2) \geq \trace(a)
\end{align*}
(this may be done without loss of generality, as otherwise we may find an equivalent totally positive element $a^\prime$ such that this holds). Given that $\|\Log(u)\|_{\infty}=\lambda_1^{\infty}$ by construction, by a similar argument as the one posed in the previous section that lead to us constructing the convex body in \ref{cube}, we are looking for integer solutions to the inequality
\begin{align*}
\|x\mathbf{b}_1+y\mathbf{b}_2\|_{\infty} \leq \frac{1}{2}\log(1+(\sigma_i(u)^2-1)^2)=\frac{1}{2}\log(1+(\exp(2\lambda^{\infty})-1)^2),
\end{align*}
which gives
\begin{align*}
\frac{\|x\mathbf{b}_1+y\mathbf{b}_2\|_{\infty}}{\lambda^{\infty}} \leq \frac{1}{2}\log((1+(\exp(2\lambda^{\infty})-1)^2)^{1/\lambda^{\infty}}).
\end{align*}
The right hand side of the inequality tends to $2$ as $\lambda^{\infty} \to \infty$, and so the solutions $(x,y)$ to the above equation have absolute values that are limited to those in the set $S$ in \ref{set1}.
\section{A Linear Complexity Reduction Algorithm}
Whilst usually we say that a form $axx^*$ is reduced if $a \in \mathcal{F}_{K_{\mathbb{R}}}$, the notion of reduction can also be defined more broadly. For example, types of reduction of (rational) quadratic forms include Minkowski \cite{geometriederzahlen}, Korkin-Zolotarev \cite{HKZ} and Lenstra-Lovasz-Lovasz (LLL) \cite{LLL}. We present a very simple algorithm that, given some totally positive element $a \in K_{\mathbb{R}}$, finds an equivalent totally positive element $a^\prime$ with ``desirable'' properties.
\begin{algorithm}\label{reductionalgorithm}
\SetKwInOut{Input}{input}\SetKwInOut{Output}{output} \Input{A totally positive element $a \in K_{\mathbb{R}}$, a unit $u \in \mathcal{O}_K^*$ and its associated elements $v_i \in \mathcal{O}_{K_\mathbb{R}}^*$, where $v_i$ is the result of embedding $\sigma_i(u)$ into $K_{\mathbb{R}}$, and some real constant $0<\delta \leq 1$.} \Output{A totally positive element $a^\prime$ that is equivalent to $a$.} \BlankLine 
\nl Set $j=1$.\\
\nl If $j=r+s$, return $a$ and end the algorithm. \\
\nl If $\trace(av_jv_j*)<\delta\trace_{K/\mathbb{Q}}(a)$, set $a \to av_jv_j^*$ and return to step 1. Otherwise, set $j \to j+1$ and return to step 2.
\caption{A simple reduction algorithm.}
\end{algorithm}
\begin{proposition}
With inputs $\delta, b$ and $u$, assuming that $\delta$ is strictly less than 1, algorithm \ref{reductionalgorithm} performs at most $\mathcal{O}(\log(X)((r+s+1)\log(X)+(r+s)\log(\max_i |\sigma_i(u)|)))$ bit operations, where $X=\max_i a_i$, where $a=(a_1,\dots,a_{r+s})$.
\end{proposition}
\begin{proof}
Clearly each full round of the algorithm either results in the termination of the algorithm, or we find some equivalent $a^\prime$ such that $\trace(a^\prime)<\delta \trace(a)$, and so since $\delta<1$ the algorithm can only perform $\mathcal{O}(\log(X))$ rounds. In the worst case, a full round would require us to compute the value of $\trace(av_jv_j^*)$ $r+s$ times. The values of $a_i$ are bounded above by $X$ and the values of $|\sigma_j(u)|$ are bounded above by $\max_j |\sigma_j(u)|$, so a round can have a maximum of $\mathcal{O}((r+s)\log(\max_i |\sigma_i(u)|)+(r+s+1)\log(X))$ bit computations.
\end{proof}
\begin{theorem}
Suppose that algorithm \ref{reductionalgorithm} takes as input a totally positive element $a=(a_1,a_2,\dots,a_{r+s}) \in K_{\mathbb{R}}$, a Pisot unit $u$, and some parameter $0<\delta \leq 1$, and outputs some $a^\prime=(a_1^\prime,a_2^\prime,\dots,a_{r+s}^\prime)$ equivalent to $a$. Denote by
\begin{align*}
\mu(a) \triangleq \min_{x \in \mathcal{O}_{K_{\mathbb{R}}}\setminus \{0\}} \trace(axx^*).
\end{align*}
Then
\begin{align}
\trace(a^\prime) \leq \max\left\{\frac{t_K(u,\delta)^2}{\min_{x \in \mathcal{S}} \trace(xx^*)},1\right\}\mu(a), \label{ineq13}
\end{align}
where 
\begin{align*}
t_K(u,\delta)=\sqrt{1+\frac{|u|^2-\delta}{\delta-\max_{j \neq 1}|\sigma_j(u)|^2}},
\end{align*}
and $\mathcal{S}$ denotes the elements of $\mathcal{O}_{K_\mathbb{R}}$ that do not correspond to roots of unity or zero in $\mathcal{O}_K$. Moreover, if
\begin{align*}
\mu(a)=\trace(axx^*)
\end{align*}
for some $x \in \mathcal{O}_{K_\mathbb{R}}$, then
\begin{align}
\trace(xx^*) \leq t_K(u,\delta)^2. \label{ineq14}
\end{align}
\end{theorem}
\begin{proof}
First, by an argument similar to that in Lemma \ref{pisotred}, each $a_i^\prime$ for $1 \leq i \leq r+s$ must satisfy
\begin{align*}
a_i^\prime \geq \trace(a^\prime)t_K(u,\delta)^{-2}.
\end{align*}
Suppose that $x=(x_1,\dots,x_{r+s}) \in \mathcal{O}_{K_\mathbb{R}}$ is the element that satisfies $\trace(a^\prime xx^*)=\mu(a)$. Clearly, inequality \ref{ineq13} holds if $x$ corresponds to a root of unity in $\mathcal{O}_K$, so we assume that $x \in \mathcal{S}$. Then
\begin{align*}
\mu(a)=\trace(a^\prime xx^*) = \left(\sum_{j=1}^{r}+ 2\sum_{j=r+1}^{r+s}\right) a_i^\prime x_ix_i^* \geq t_K(u,\delta)^{-2} \trace(a^\prime)\left(\sum_{j=1}^{r}+ 2\sum_{j=r+1}^{r+s}\right) x_ix_i^*.
\end{align*}
Inequalities \ref{ineq13} and \ref{ineq14} follow.
\end{proof}

\section*{Appendix}
\begin{proof}[Proof of Lemma \ref{n=3lemma}]
Throughout this proof, we will let $\mathbf{v}=x\mathbf{b}_1+y\mathbf{b}_2$ for some integer pair $(x,y) \in \mathbb{Z}^2$. Our method of proof is to assume that we have a pair $(x,y)$ that is not contained in $S$ such that $\|\mathbf{v}\|_{\infty} <2\lambda^{\infty}$, such that $\mathbf{b}_1,\mathbf{b}_2$ obey the inequalities \ref{ineq1}, and to show that we come by a contradiction. We omit the case where either $x$ or $y$ is zero, since obviously then we will come by a contradiction. We may freely assume throughout that $\|\mathbf{b}_1\|_{\infty}=\alpha>0$. We will assume throughout that $\|\mathbf{b}_2\|_{\infty}<2\|\mathbf{b}_1\|_{\infty}$, as otherwise, since it was shown in \cite{3} under the assumption that inequalities \ref{ineq1} holds we have $\|x\mathbf{b}_1+y\mathbf{b}_2\|_{\infty} \geq \|\mathbf{b}_2\|$ for all nonzero $x,y$, we would come to a contradiction.

Throughout the next part of the proof, we will assume that $\|\mathbf{b}_2\|_{\infty}=\gamma>0$, as we can force $\gamma$ to be positive by a switch of signs. By \ref{ineq1}, at least one of the following inequalities holds:
\begin{align}
&|\alpha-\gamma| \geq \|\mathbf{b}_2\|_{\infty}, \label{ineq2},
\\& |\beta-\delta| \geq \|\mathbf{b}_2\|_{\infty}, \label{ineq3}
\\& |\alpha+\beta-\gamma-\delta| \geq \|\mathbf{b}_2\|_{\infty}. \label{ineq4}
\end{align}
Clearly, since $\gamma>0$ and $\gamma=\|\mathbf{b}_2\|<2\alpha$, clearly \ref{ineq2} cannot hold. So, the inequalities we have to possibly consider are
\begin{align}
&|\beta-\delta| \geq \gamma, \label{ineq5},
\\& |\alpha+\beta-\gamma-\delta| \geq \gamma. \label{ineq6}
\end{align}
Assume that $\label{ineq5}$ holds. Since $|\alpha+\beta| \leq \alpha$, we must have $\beta<0$, so necessarily we must also have that $\delta>0$. But $|\gamma+\delta| \leq \gamma$ also implies that $\delta<0$, which is a contradiction. So instead assume that \label{ineq6} holds. Then
\begin{align*}
\max\{\alpha+\beta-\gamma-\delta,\gamma+\delta-\alpha-\beta\} \geq \gamma.
\end{align*}
But this cannot hold:
\begin{align*}
\alpha+\beta-\gamma-\delta \geq \gamma \iff \alpha>\alpha+\beta \geq 2\gamma+\delta,
\end{align*}
but $|\gamma+\delta| \leq \gamma$ implies that $\delta<0$, so this would imply that $\alpha>\gamma$ which is a contradiction, and
\begin{align*}
\gamma+\delta-\alpha-\beta \geq \gamma \iff \delta \geq \alpha+\beta>0,
\end{align*}
but again we need $\delta<0$, which is a contradiction. Hence, it cannot hold that $\gamma=\|\mathbf{b}_2\|_{\infty}$.

Throughout the next part of the proof, we will assume that $\|\mathbf{b}_2\|_{\infty} =\delta >0$, as we can force $\delta$ to be positive by a switch of signs, and we stop assuming any sign on $\gamma$. In fact, $|\gamma+\delta| \leq \delta$ necessarily implies that $\gamma<0$ now. By \ref{ineq1}, at least one of the following inequalities hold:
\begin{align}
&|\alpha+\gamma| \geq \|\mathbf{b}_2\|_{\infty}, \label{ineq7}
\\& |\beta+\delta| \geq \|\mathbf{b}_2\|_{\infty}, \label{ineq8}
\\& |\alpha+\beta+\gamma+\delta| \geq \|\mathbf{b}_2\|_{\infty}. \label{ineq9}
\end{align}
Since $\gamma<0$ and $\alpha \leq \delta$, clearly \ref{ineq7} cannot hold. Moreover, clearly \ref{ineq8} cannot hold since necessarily $\beta<0$. So the only case we need to consider here is
\begin{align*}
|\alpha+\beta+\gamma+\delta| \geq \delta.
\end{align*}
Since $\alpha \geq -\beta$ and $\delta \geq -\gamma$,
\begin{align*}
\alpha+\beta+\gamma+\delta \geq \delta \iff  \alpha \geq -\beta-\gamma
\end{align*}
Hence, if $x,y$ are of opposite signs,
\begin{align*}
\|\mathbf{v}\|_{\infty} \geq \max\{|x|\alpha+|y||\gamma|, |x||\beta|+|y|\delta\},
\end{align*}
which is greater than or equal to $2\lambda^{\infty}$ if $\min\{|x|,|y|\} \geq 2$, so we assume that $x,y$ are of the same sign. Then
\begin{align*}
\|\mathbf{v}\|_{\infty} \geq ||x|(\alpha+\beta)+|y|(\gamma+\delta)| \geq |x|\alpha+|y|\delta-\max\{|x|,|y|\}\alpha \geq \min\{|x|,|y|\}\alpha,
\end{align*}
which is greater than or equal to $2\lambda^{\infty}$ if $\min\{|x|,|y|\} \geq 2$, so this would give us another contradiction when we assume that $\|\mathbf{b}_2\|_{\infty}=\delta$.

For the remainder of the proof, we will assume that $\|\mathbf{b}_2\|_{\infty}=\gamma+\delta>0$, as we can force the sign of $\gamma+\delta$ to be positive by a switch of signs, and we stop assuming a sign on $\delta$. Since $\alpha+\beta>0$, clearly now \ref{ineq4} cannot hold, so we need to consider the possible inequalities
\begin{align}
&|\alpha-\gamma| \geq \gamma+\delta, \label{ineq10}
\\&|\beta-\delta| \geq \gamma+\delta.  \label{ineq11}
\end{align}
In fact, note that $\gamma,\delta$ must be the same sign, as otherwise we would have $|\gamma+\delta|<\min\{|\gamma|,|\delta|\}$. Therefore, $\gamma,\delta>0$ as $\gamma+\delta>0$ by assumption. Therefore, \ref{ineq10} cannot hold as $|\alpha-\gamma|<\alpha$, since $\gamma<\gamma+\delta < 2\alpha$ by assumption. So the only possibility is
\begin{align}
|\beta|+\delta \geq \gamma+ \delta \iff |\beta| \geq \gamma. \label{ineq12}
\end{align}
If $x,y$ are the same sign,
\begin{align*}
\|\mathbf{v}\| \geq \max\{|x|\alpha,|y|(\gamma+\delta)\}
\end{align*}
which is greater than or equal to $2 \lambda^{\infty}$ if $\min\{|x|,|y|\} \geq 2$, so assume that $x,y$ are of opposite signs. Then
\begin{align*}
\|\mathbf{v}\| \geq |x||\beta|+|y|\delta \geq |x|\gamma+|y|\delta,
\end{align*}
which is greater than or equal to $2\lambda^{\infty}$ if $\min\{|x|,|y|\} \geq 1$. Therefore assume that $\min\{|x|,|y|\}=1$. If $\min\{|x|,|y|\}=|x|=1$, then
\begin{align*}
\|\mathbf{v}\| \geq |y|(\gamma+\delta)-(\alpha+\beta) \geq (|y|-1)\alpha,
\end{align*}
which is greater than or equal to $2\lambda^{\infty}$ if $|y| \geq 3$. Finally, assume that $\min\{|x|,|y|\}=|y|=1$. Then by \ref{ineq12},
\begin{align*}
\|\mathbf{v}\| \geq |x|\alpha -\gamma \geq |x|\alpha -|\beta| \geq (|x|-1)\alpha,
\end{align*}
which is greater than or equal to $2\lambda^{\infty}$ if $|x| \geq 2$. Therefore, we have arrived at a contradiction for every case, and so the lemma holds.
\end{proof}
\begin{lemma}\label{sumbound}
$\displaystyle\sum_{k=0}^{\left\lfloor\frac{r+s}2\right\rfloor}(-1)^k\binom{r+s}k\left(\frac{r+s}2-k\right)^{r+s-1}\le\left(\frac{\sqrt e}{2\pi}+o(1)\right)\left(e^{1+\frac1{2e}}\right)^{r+s}(r+s-1)!$ as $r+s\to+\infty$.
\end{lemma}
\begin{proof} As we're treating $r+s$ as a single variable, we will define $n=r+s$ as a substitute. The left-hand side may now be written as $\left|\sum\limits_{k=0}^{\lfloor n/2\rfloor}(-1)^k\binom nk\left(\frac n2-k\right)^{n-1}\right|$. The alternating sum makes it quite tedious, so from there on, we will split the sum in even and odd parts, giving us
$$\left|\sum\limits_{k=0}^{\lfloor n/2\rfloor}(-1)^k\binom nk\left(\frac n2-k\right)^{n-1}\right|=\left|\underset{k~\rm even}{\sum_{k=0}^{\lfloor n/2\rfloor}}\binom nk\left(\frac n2-k\right)^{n-1}-\underset{k~\rm odd}{\sum_{k=1}^{\lfloor n/2\rfloor}}\binom nk\left(\frac n2-k\right)^{n-1}\right|$$

To figure out whether the odd (negative) or even (positive sum is greater in the end, we can figure out that, from {\bf Lemma 5}, $\sum\limits_{k=0}^{\lfloor n/2\rfloor}(-1)^k\binom nk\left(\frac n2-k\right)^{n-1}=\vol(\mathcal C(R)\cap V)\cdot\frac{(n-1)!}{R^{n-1}\sqrt n}\ge0$, so :
$$\sum_{k=0}^{\lfloor n/2\rfloor}(-1)^k\binom nk\left(\frac n2-k\right)^{n-1}\le\underset{k~\rm even}{\sum_{k=0}^{\lfloor n/2\rfloor}}\binom nk\left(\frac n2-k\right)^{n-1}$$

As our first operation, we will expand $\binom nk$ as $\frac{n!}{k!\,(n-k)!}$. As these are positive sums, an asymptotic bound of their summand will yield a sum which has a similar asymptotic relation.

Indeed, for positive-valued functions $f$ and $g$ over the integers, $f(n)=o(g(n))$ implies $\sum_nf(n)=o\left(\sum_ng(n)\right)$, and $f(n)\sim cg(n)\implies\sum_nf(n)=(c+o(1))\sum_ng(n)$. We will use this property on the even sum using Stirling's formula, which states that $n!\sim\sqrt{2\pi n}\left(\frac ne\right)^n$. We will expand the $(n-k)!$ term from the $\frac{n!}{k!\,(n-k)!}$ formula as its asymptotic behavior, according to Stirling :
$$\begin{array}{rcl}
\displaystyle\underset{k~\rm even}{\sum_{k=0}^{\lfloor n/2\rfloor}}\frac{n!}{k!\,(n-k)!}\left(\frac n2-k\right)^{n-1}
&\sim&\displaystyle\frac{e^n\cdot n!}{\sqrt{2\pi}}\underset{k~\rm even}{\sum_{k=0}^{\lfloor n/2\rfloor}}\frac1{k!\,(n-k)^{n-k+1/2}}\left(\frac1e\right)^k\left(\frac n2-k\right)^{n-1}
\end{array}$$

In order to get something a bit simpler to handle, we will upper bound the summand to figure out a simpler expression. We notice that $n/2\le n-k\le n$. Therefore :
$$\displaystyle\frac{e^n\cdot n!}{\sqrt{2\pi}}\underset{k~\rm even}{\sum_{k=0}^{\lfloor n/2\rfloor}}\frac1{k!\,(n-k)^{n-k+1/2}}\left(\frac1e\right)^k\left(\frac n2-k\right)^{n-1}
\le\frac{e^n\cdot n!}{\sqrt{2\pi}(n/2)^{3/2}}\underset{k~\rm even}{\sum_{k=0}^{\lfloor n/2\rfloor}}\frac1{k!}\left(\frac n{2e}\right)^k$$

We will try to find a supremum to the summand. For this, we will compute the derivative of $x\mapsto(n/(2e))^x/\Gamma(x+1)$ :
$$\frac\partial{\partial x}\left(\frac1{\Gamma(x+1)}\left(\frac n{2e}\right)^x\right)=\frac1{\Gamma(x+1)}\left(\frac n{2e}\right)^x\left(\log\left(\frac n{2e}\right)-\psi(x+1)\right),$$
where $\psi$ is the digamma function. $\psi$ is a strictly increasing function over $(0,+\infty)$, and spans all of $\mathbb R$ within this interval, which means it has precisely one root. Since $\log\frac n{2e}\to+\infty$ as $n\to+\infty$, that $\psi(x)=\log(x)+o(1)$, and that more precisely $\exp\psi(x)=x+O(1)$, we can infer that the root of this derivative is asymptotic to $\frac n{2e}$ as $n\to+\infty$.

Hence, we can derive the following relation :

$$\begin{array}{rcl}
\displaystyle\frac{e^n\cdot n!}{\sqrt{2\pi}(n/2)^{3/2}}\underset{k~\rm even}{\sum_{k=0}^{\lfloor n/2\rfloor}}\frac1{k!}\left(\frac n{2e}\right)^k
&\le&\displaystyle\frac{e^n\cdot n!}{\sqrt{2\pi}(n/2)^{3/2}}\underset{k~\rm even}{\sum_{k=0}^{\lfloor n/2\rfloor}}\sup_{k'\in[0,n/2]}\left\{\frac1{k'!}\left(\frac n{2e}\right)^{k'}\right\}\\[20pt]
&\sim&\displaystyle\frac{e^n\cdot n!}{\sqrt{2\pi}(n/2)^{3/2}}\underset{k~\rm even}{\sum_{k=0}^{\lfloor n/2\rfloor}}\frac1{\Gamma(\frac n{2e}+1)}\left(\frac n{2e}\right)^{\frac n{2e}}\\[20pt]
&=&\displaystyle\frac{e^n\cdot n!}{\sqrt{2\pi}(n/2)^{3/2}}\left\lfloor\frac n4\right\rfloor\,\frac1{\Gamma(\frac n{2e}+1)}\left(\frac n{2e}\right)^{\frac n{2e}}
\end{array}$$

Using Stirling's approximation once again, this time as $\Gamma(x+1)\sim\sqrt{2\pi x}(x/e)^x$, as well as $\lfloor x\rfloor\sim x$,
$$\begin{array}{rcl}
\displaystyle\frac{e^n\cdot n!}{\sqrt{2\pi}(n/2)^{3/2}}\left\lfloor\frac n4\right\rfloor\,\frac1{\Gamma(\frac n{2e}+1)}\left(\frac n{2e}\right)^{\frac n{2e}}
&\sim&\displaystyle\frac{\sqrt e}{2\pi}\left(e^{1+\frac1{2e}}\right)^n(n-1)!
\end{array}$$
This concludes the proof.
\end{proof}
\end{document}